\documentclass[11pt,xcolor=dvipsnames,svgnames,table,reqno]{amsart}

\usepackage[utf8]{inputenc}
\usepackage{graphicx}
\graphicspath{ {./images/} }
\usepackage{hyperref}
\usepackage{amssymb} 
\usepackage{comment}
\usepackage{amsthm} 
\usepackage{mathtools}
\usepackage{nccmath}
\usepackage{physics} 
\usepackage{tipa} 
\usepackage{yfonts}
\usepackage{enumerate} 
\usepackage[T1]{fontenc} 
\usepackage[inline]{enumitem} 
\usepackage{xspace} 
\usepackage{commath}
\usepackage{tikz-cd}
\usepackage{color} 
\usepackage{pgfplots}
\pgfplotsset{compat=1.18}

\usepackage[left=1.5in, right=1.5in, bottom=1.25in,
height=8in]{geometry}
\usepackage{wrapfig}

\usepackage{tikz}
\usetikzlibrary{shapes,shadows,arrows}
\usepackage{tikz-cd}
\usetikzlibrary{cd}

\newcommand{\unstable}[1]{W^u_{\delta}(#1)}
\newcommand{\stable}[1]{W^s_{\delta}(#1)}
\newcommand{\tope}[1]{\mathrm{h}_{\mathrm{top}}(#1)}
\newcommand{\be}[1]{\hbar^{HW}(#1)}

\newtheorem{definition}{Definition}
\newtheorem{theorem}{Theorem}

\newtheorem{lemma}{Lemma}

\newtheorem{remark}{Remark}
\newtheorem{proposition}{Proposition}
\newtheorem*{customtheorem}{Theorem A} 
\newtheorem*{thmB}{Theorem B}

\title{\textbf{WRAPPED FLOER HOMOLOGY AND HYPERBOLIC SETS}}
\author{Rafael A. Fernandes}
\address{Department of Mathematics, UC Santa Cruz, Santa
  Cruz, CA 95064, USA} 
\email{rfernan9@ucsc.edu}

\begin{document}
\keywords{Periodic orbits, Reeb flows, wrapped Floer homology, topological
  entropy, barcode entropy, persistence modules}

\date{\today} 

\thanks{The work is partially supported by the NSF grant DMS-2304206}

\maketitle

\begin{abstract}
   In this paper, we continue the quest to understand the interplay between wrapped Floer homology barcode and topological entropy. Wrapped Floer homology barcode entropy is defined as the exponential growth, with respect to the left endpoints, of the number of not-too-short bars in its barcode. We prove that, in the presence of a topologically transitive, locally maximal hyperbolic set for the Reeb flow on the boundary of a Liouville domain, the barcode entropy is bounded from below by the topological entropy restricted to the hyperbolic set.
\end{abstract}
\tableofcontents
\section{Introduction and main results}
\subsection{Introduction}

In this paper, we further explore the connections between Floer homology and the dynamics of the underlying system. Our focus is on understanding how the hyperbolic dynamics of a Reeb flow influences barcode entropy. More precisely, we prove that in the presence of a topologically transitive, locally maximal hyperbolic set for the Reeb flow on the boundary of a Liouville domain, the wrapped Floer homology barcode entropy is bounded from below by the topological entropy of the Reeb flow restricted to the hyperbolic set. 

The barcode entropy of a persistence module is an invariant roughly defined as the exponential growth, with respect to the left endpoint, of the number of bars that are not-too-short in its barcode. In the case of a sequence of persistence modules indexed by the integers, the barcode entropy is the exponential growth, with respect to the indexes, of the number of bars that are not-too-short.

In \cite{cineli2021topological}, Çineli, Ginzburg And Gürel defined the notion of barcode entropy $\hbar(\varphi)$ for a Hamiltonian diffeomorphism $\varphi$ using the sequence $HF(\varphi^k)$, and the relative barcode entropy $\hbar(\varphi, L_0,L_1)$, where $L_0$ and $L_1$ are Lagrangians, via the sequence $HF(L_0,\varphi^k(L_1))$. Among other things, they proved that in the presence of a locally maximal hyperbolic set $K$ for $\varphi$, the barcode entropy is bounded below by the topological entropy of $\varphi$ restricted to the hyperbolic set $K$, i.e., $\hbar(\varphi) \geq \tope{\varphi|_K}$, a result they referred to as Theorem B. Later, in \cite{meiwes2024barcode}, Meiwes proved a relative version of Theorem B. Specifically, for a Hamiltonian diffeomorphism $\varphi$ and a pair of Hamiltonian isotopic Lagrangians $(L_0, L_1)$ containing the local stable and local unstable manifolds of two points in $K$, i.e., $W_{\delta}^u(q) \subset L_0$ and $W_{\delta}^s(p) \subset L_1$, for some $q, p \in K$ and $\delta > 0$, it holds that $\hbar(\varphi, L_0, L_1) \geq \tope{\varphi|_K}$. In the context of Reeb flows, specifically those on the boundary of Liouville domains, Fender, Lee, and Sohn defined the notion of barcode entropy in \cite{fender2023barcode} via filtered symplectic homology, and subsequently, in \cite{cineli2024barcode}, Çineli, Ginzburg, Gürel, and Mazzucchelli proved a version of Theorem B in this setting. More precisely, they showed that in the presence of a locally maximal hyperbolic set $K$ for the contact form $\alpha = \lambda|_M$ on the boundary $\Sigma$ of a Liouville domain $(M, \lambda)$, we have $\hbar(\alpha) \geq \tope{K}$, where $\tope{K} := \tope{\varphi^t_{\alpha}|_{K}}$. In this paper, we prove a relative version of Theorem \hyperref[thmB]{B} for the Reeb flow case. 

Notice that Theorem \hyperref[thmB]{B} indicates that the barcode entropy captures the local behavior of the underlying dynamical system, rather than just the global topology of the underlying space. There are several results relating the topological entropy of geodesic flows and the global topology of the underline manifolds, as for example in \cite{dinaburg1971relations,katok1982entropy,paternain2012geodesic}. Analogs of these ideas were brought to the contact case, and similar results were obtained such as in, for example, \cite{abbondandolo2023entropy,alves2016cylindrical,alves2019legendrian,alves2019topological,alves2022c,alves2022reeb,alves2023c,macarini2011positive}. What distinguishes the results presented here from previous works is their independence from the global properties of the underlying system and space. Specifically, our results do not depend on the isotopy class of the map, the exponential growth of the Floer homology, or the topology of the configuration/phase space. For example, the wrapped Floer homology barcode entropy may be positive independently of the growth of $\pi_1(M)$ or $H_*(\Lambda)$, e.g., when $M$ is a sphere or torus and the underling flow admits a hyperbolic set with positive topological entropy.

\subsection{Main results}
Consider a Liouville domain $(M,\lambda)$, and denote $\alpha = \lambda|_{\Sigma}$ the restriction of $\lambda$ to the boundary $\Sigma = \partial M$. Let $L_0$ and $L_1$ be exact asymptotically conical Lagrangians in $M$.
We denote the filtered wrapped Floer homology of $(M,L_0 \rightarrow L_1)$ on the interval $(-\infty,t)$ by $HW^t(M,L_0 \rightarrow L_1)$. Together with the "inclusion maps", the family of vector spaces $t \rightarrow HW^t(M,L_0 \rightarrow L_1)$ form a persistence module in the sense of Section 2 in \cite{fernandes2024barcode}, and we can therefore consider its barcode. 

For $\epsilon>0$, let $\text{\textcrb}_{\epsilon}(M,L_0\rightarrow L_1,t)$ the number of bars with a length greater than or equal to $\epsilon$ and with left end point at most $t$. Note that this number increases with $t$ and $1/\epsilon$. The \textit{wrapped Floer homology barcode entropy} is defined as follows. (Here $\log^+(x) = \log(x)$, for $x > 0$, and $\log^+(0)=0$.)

\begin{definition}
    The wrapped Floer homology \textbf{barcode entropy} of $M$ is defined by 
    $$\hbar^{HW}(M,L_0 \rightarrow L_1) =\lim_{\epsilon \rightarrow 0} \limsup_{t \rightarrow \infty} \frac{\log^{+}\text{\textcrb}_{\epsilon}(M,L_0 \rightarrow L_1,t)}{t},$$
\end{definition}

We point out that this definition makes sense in a broader sense, i.e., one can define the barcode entropy of any persistence module. One could ask whether the barcode entropy is not a trivial invariant (not always zero). From the work of Meiwes in \cite{meiwes2018rabinowitz}, we see that there exist examples where the symplectic growth of wrapped Floer homology, i.e. the exponential growth of the number of infinity bars with respect to the left end point is positive, therefore, so is the barcode entropy. Therefore, the barcode entropy is a non-trivial invariant. Moreover, for a fillable contact manifold, the wrapped Floer homology barcode entropy is independent of the filling; see \cite{fernandes2024barcode} for details.

In what follows, $K$ is a hyperbolic set for $\alpha$, and $W^s_{\delta}(p)$ and $W^u_{\delta}(q)$ denote the local stable and local unstable manifolds of $p,q \in K$ respectively. For a better understanding of the dynamical concepts involved in the following theorem, we refer to \cite{fisher2019hyperbolic,katok1995introduction}. In what follows, we set
$\tope{K} := \tope{\varphi^t_{\alpha}|_{K}}$.

\begin{thmB} \label{thmB}
Let $K$ be a compact, topologically transitive hyperbolic invariant set of the Reeb flow $\varphi^t_{\alpha}$. If $W^s_{\delta}(p) \subset \Lambda_0$ and $W^u_{\delta}(q) \subset \Lambda_1$, for some $q,p \in K$, and $\delta>0$, then
\begin{equation} \label{eq:barcode entropy bounded by below}
\be{M,L_0 \rightarrow L_1} \geq \tope{K}.
\end{equation}
\end{thmB}

One way to understand the above theorem is that the barcode entropy on the boundary of a Liouville domain captures the hyperbolic dynamics of the Reeb flow. 

Theorem \hyperref[thmB]{B} can be thought as the relative version of Theorem B in \cite{ginzburg2014hyperbolic}, and also as the natural follow-up question after the proof Theorem A, stated below by the sake of completeness, and proved in \cite{fernandes2024barcode}, towards understanding the wrapped Floer homology barcode entropy. 

\begin{customtheorem}[\cite{fernandes2024barcode}] \label{theoremA} Let $(M,\lambda)$ be a Liouville domain, and $L_0$, and $L_1$ be connected exact asymptotically conical Lagrangians in $M$. Then
\begin{align} 
    \label{equationentropies}
    \hbar^{HW}(M,L_0 \rightarrow L_1) \leq h_{top}(\alpha). 
\end{align}
\end{customtheorem}

From the above theorem we observe that the barcode entropy, a Floer-theoretic invariant, is bounded from above by the topological entropy, a topological invariant.

Regarding the assumptions of the Theorem \hyperref[thmB]{B}, It is not hard to see that given any point in a hyperbolic set $K$ for $\alpha$, we can always find a Legendrian that locally coincides with the local stable or unstable manifold of that point. Indeed, we can always move the Legendrians making them pass through a point, and modify them locally. 

The proof of Theorem \hyperref[thmB]{B} relies heavily on a relative, parametric version of the \textit{Crossing Energy Theorem}. In this context, it states that for a convex semi-admissible Hamiltonian $H$, any small energy Floer strip for $(H,L_0 \rightarrow L_1)$, where $L_0$ and $L_1$ are exact asymptotically conical Lagrangians, asymptotic to Hamiltonian chords at both ends, one of which is a $(q,p,\delta)$-chord, with $q,p \in K$, and $W^s_{\delta}(p) \subset \Lambda_0 = \partial L_0$ and $W^u_{\delta}(q) \subset \Lambda_1 = \partial L_1$, has energy bounded from below by some $\sigma>0$. Various versions of this theorem have been established in \cite{ginzburg2014hyperbolic,ginzburg2018hamiltonian,cineli2021topological,meiwes2024barcode,cineli2023invariant,cineli2024barcode}. For Liouville domains, the proof of the Crossing Energy Theorem was made possible by the \textit{Location Constraint Theorem}, first proved in \cite{cineli2023invariant}. For completeness, we present a relative version of the Location Constraint Theorem and use it to derive a relative version of the Crossing Energy Theorem.

We tend to believe that the two invariants do not coincide, as shown by Çineli in the case of Hamiltonian Floer homology \cite{cineli2023generalized}. On the other hand, we expect the two invariants to satisfy a maximal principle, i.e.,
$$\sup_{L_0, L_1}\{\be{M,L_0 \rightarrow L_1}\} = \tope{\alpha},$$
where the supremum is taken over all pairs of connected exact asymptotically conical Lagrangians. \\

This paper is organized as follows. In the first section, we review the construction of filtered wrapped Floer homology and some facts about Floer strips. The second section revisits the definition of wrapped Floer homology barcode entropy and introduces an alternative method for computing it. In the final section, which is divided into three parts, we present the proof of Theorem B. The first part explores the relationship between topological entropy and the exponential growth of periodic orbits. In the second part, we introduce the main ingredient used in the proof of Theorem B and provide its proof. Finally, in the third part, we conclude by proving the intermediate theorem necessary for the proof of Theorem B.

\subsection*{Acknowledgments} 
 \addtocontents{toc}{\protect\setcounter{tocdepth}{-1}} The author is deeply grateful to Viktor Ginzburg for his invaluable guidance throughout this project and to Brayan Ferreira for useful discussions. Part of this work was conducted at the Federal University of Espírito Santo, Brazil, during a research visit in the summer of 2024. We would like to thank the institution for its warm welcome and hospitality.
 \addtocontents{toc}{\protect\setcounter{tocdepth}{2}}

\section{Wrapped Floer homology} In this section, we discuss the basic construction of filtered wrapped Floer homology and recall some facts that will be useful for proving the main theorems of this paper.
\subsection{Exact conical Lagrangians} This subsection provides the definitions of the Lagrangians we consider in wrapped Floer homology, along with some remarks that will be useful. 

Let $(M,\lambda)$ be a Liouville domain, and $(L,\partial L) \subset (M,\partial M)$ a Lagrangian. We call $L$ \emph{asymptotically conical} if 
\begin{itemize}
    \item $\Lambda = \partial L$ is a Legendrian submanifold of $(\Sigma, \xi_{M})$, where $\Sigma = \partial M$, and $\xi_{M}$ is the contact structure induced by $\lambda|_{\Sigma} = \alpha$,
    \item $L \cap [1-\epsilon,1] \times \Sigma = [1-\epsilon,1] \times \Lambda$ for sufficiently small $\epsilon > 0$.
 \end{itemize}
We can extend $L$ to an Lagrangian $\widehat{L} \subset \widehat{M}$ by taking $\widehat{L} = L \cup_{\Lambda} [1,\infty)\times \Lambda$, where $\widehat{M}$ denotes the symplectic completion of $M$, i.e.,
$$\widehat{M} = M \cup_{\Sigma
} [1,\infty) \times \Sigma,$$
with the symplectic form $\omega=d\lambda$ extended to $[1,\infty)\times \Sigma$ as
$$\omega = d(r\alpha),$$
where $r$ is the coordinate in $[1,\infty)$.
We also call the extended Lagrangians in $\widehat{M}$ asymptotically conical. In what follows, we will deal with \emph{exact asymptotically conical} Lagrangians, i.e., Lagrangians $(L,\partial L) \subset (M,\Sigma)$ that are asymptotically conical, and $\lambda|_{L} = df$, for some $f \in C^{\infty}(L)$. We refer to Lagrangians satisfying the last condition as \emph{exact Lagrangians}. Throughout the rest of the paper, we assume that all the Legendrian and Lagrangian submanifolds considered are connected.
\begin{remark} \label{intersectionpointshaveaction0}
    For an exact asymptotically conical Lagrangian $L$, the primitive $f$ for $\lambda_{L}$ can be taken identically zero by a modification of $\lambda$ without changing the dynamics of the Reeb flow on the boundary $\partial M$. Indeed, since $f$ is constant in a neighborhood of the boundary, by adding a constant to it we can assume the constant is zero, i.e., $f\equiv 0$ on that neighborhood. Now we extend $f$ to a function $F$ on $M$ with compact support in M minus the collar neighborhood. By adding $-dF$ to $\lambda$, we obtain a new Liouville form, which we still denote by $\lambda$, so that $\lambda|_{L}=0$. This procedure changes $\lambda$ in the interior of $M$ but not at the boundary. In the case of a pair of exact asymptotically conical Lagrangians $(L_{0},L_{1})$, we cannot use the same argument to change $\lambda$ in such a way that $\lambda|_{L_{0}}=\lambda|_{L_{1}} \equiv 0$, but we can assume $f_{0}(x)=f_{1}(x)$, for all intersection points $x\in L_0 \cap L_1$, when $L_0 \pitchfork L_1$, by modifying $\lambda$ in a neighborhood of each intersection point. Notice that since $L_0 \pitchfork L_1$, there are no intersection on the collar and therefore this modification does not effect the dynamics on the boundary. 
\end{remark}   
\subsection{Hamiltonians and action filtration}
Now we present the basic construction of wrapped Floer homology and define the persistence module obtained from it. We mostly follow \cite{meiwes2018rabinowitz}, \cite{ritter2013topological}, \cite{cineli2023invariant} and \cite{fernandes2024barcode}.

Let $(M,\lambda)$ be a Liouville domain, and fix $L_{0}$, $L_{1}$ exact asymptotically conical Lagrangians in $M$ with $\Lambda_{0}=\partial L_{0}$ and $\Lambda_{1}= \partial L_{1}$.  

A Hamiltonian $H: \widehat{M} \rightarrow \mathbb{R}$ is called \emph{admissible} if 
\begin{itemize}
    \item $H<0$ on $M$, and
    \item There exists $T>0$, $B \in \mathbb{R}$, and $r_0 \geq 1$ such that, $H(r,x)=rT - B$ on $[r_0,\infty)\times \partial M$.
\end{itemize}
We call $T$ the \textit{slope of $H$ at infinity}  and denote it by $slope(H)=T$. When H satisfies only the last condition, we call it \textit{linear at infinity}. In what follows, we will need to deal with Hamiltonians $H$ so that $H|_{M} \equiv 0$, and are linear at infinity. We will refer to these Hamiltonians as \textit{semi-admissible}. It will be convenient to consider the class of  \emph{convex (semi-)admissible} Hamiltonians. These are (semi-)admissible Hamiltonians $H$ such that
\begin{itemize}
    \item $H|_M \equiv const$ ,
    \item $H(r,x) = h(r)$, in $[1,r_{max}) \times \partial M$, for $h:[1,r_{max}] \rightarrow \mathbb{R}$ with $h''>0$ in $(1,r_{max})$,
    \item $H$ linear at infinity in $[r_{max}, \infty) \times \partial M$, with $slope(H) = T$. 
\end{itemize}

We emphasize that $r_{max}$ depends on $H$, i.e., $r_{max} = r_{max}(H)$, but for convenience, we might omit $H$ from the notation when the Hamiltonian is well understood.

We denote by $X_H$ the Hamiltonian vector field associated to a Hamiltonian $H:\widehat{M}\rightarrow \mathbb{R}$, defined by
$$\omega(X_H,\cdot) = -dH,$$
along with $\varphi_H^t$, for $t\in \mathbb{R}$, and $\varphi_H = \varphi_H^1$,  its Hamiltonian flow and Hamiltonian diffeomorphism, respectively.  

We call a Hamiltonian $H:\widehat{M} \rightarrow \mathbb{R}$ non-degenerate for $L_{0}$ and $L_{1}$ if
$$\varphi_H(\widehat{L}_{0}) \pitchfork \widehat{L}_{1}.$$

Unless specified otherwise, all the Hamiltonians in this paper are considered to be admissible.

We denote the set of smooth chords from $\widehat{L_0}$ to $\widehat{L_1}$ by $\mathcal{P}_{L_0 \rightarrow L_1}$, i.e.,
$$\mathcal{P}_{L_0 \rightarrow L_1} = \{\gamma:[0,T] \rightarrow \widehat{M} \ \mid \ \gamma(0) \in \widehat{L_0}, \ \gamma(T) \in \widehat{L_1}\}.$$
For a chord $\gamma : [0,T]\rightarrow \widehat{M} \in \mathcal{P}_{L_0 \rightarrow L_1}$, we call $T$ the \emph{length} and denote it by $length(\gamma)$. We denote by $\widehat{\mathcal{P}}_{L_0 \rightarrow L_1}$ the chords in $\mathcal{P}_{L_0 \rightarrow L_1}$ with length $1$.

For the Reeb vector field $R_\alpha$ on the boundary $(\Sigma, \xi_{M})$ associated to the contact form $\alpha$, we denote its flow by $\varphi_{\alpha}^t$, with $t \in \mathbb{R}$. The set of lengths of Reeb chords from $\Lambda_{0}$ to $\Lambda_{1}$ will be denoted by $\mathcal{S}(\alpha,\Lambda_{0}\rightarrow \Lambda_{1})$ . This is a closed, nowhere dense subset of $[0,\infty)$.

 Note that for a convex Hamiltonian on $[1,\infty) \times \Sigma$, 
$$X_H = h'(r)R_{\alpha}.$$
In this case, the Hamiltonian chords from $\widehat{L_0}$ to $\widehat{L_1}$ in $\{r\}\times \Sigma$ with length $1$, for $r \geq1$, correspond to Reeb chords of length $t= h'(r)$.

To a Hamiltonian $H$, we associate the \textit{action functional} $\mathcal{A}^{L_{0} \rightarrow L_{1}}_{H} = \mathcal{A}_{H} : \widehat{\mathcal{P}}_{L_{0} \rightarrow L_{1}} \rightarrow \mathbb{R}$ defined by 
\begin{equation} \label{eq: action functional}
\mathcal{A}_{H}(\gamma) = f_{0}(\gamma(0)) - f_{1}(\gamma(1)) + \int^{1}_{0} \gamma^{*}\lambda - \int^{1}_{0}H(\gamma(t))dt,
\end{equation}
where for $i=0,1$, the $f_{i}$ is a smooth functions on $L_{i}$ such that $df_{i} = \lambda|_{L_{i}}$.  It is not hard to see that the critical points of $\mathcal{A}_{H}$ are the Hamiltonian chords of $H$ from $\widehat{L}_{0}$ to $\widehat{L}_{1}$ with length $1$ (which we call from now on \textit{Hamiltonian chords} when the Lagrangians are well understood). We denote by $Crit(H,L_0 \rightarrow L_1)$ the set of critical points of $\mathcal{A}_{H}$. Notice that from Remark \ref{intersectionpointshaveaction0}, the action functional \eqref{eq: action functional} reduces to 
$$\mathcal{A}_{H}(\gamma) =  \int^{1}_{0} \gamma^{*}\lambda - \int^{1}_{0}H(\gamma(t))dt,$$
for chords $\gamma$ in $[1,\infty) \times \Sigma$.

For a convex semi-admissible Hamiltonian, we can associate the \textit{reparametrization function} that transforms lengths in actions. More precisely:
\begin{lemma}
    For a convex semi-admissible Hamiltonian $H$, with $H(r,x)= h(r)$ in $[1,r_{max}(H))\times \Sigma$ and $H(r,x) = rT - B$, in $r \geq r_{max}$, the reparametrization function $\widetilde{A_H} =\widetilde{A_h}$ that transforms lengths in actions given by $\widetilde{A_h}(t)=rh^{'}(r) - h(r)$, where $h'(r)=t$ is a well defined, continuous monotone increasing bijection between $[0,T]$ and $[0,B]$.
\end{lemma}
\begin{proof}
    For $1 \leq r < r' \leq r_{max}$, then
    \begin{equation} \label{eq: monotonicity of reparamentrization}
         rh'(r) -h(r)|^{r}_{r'} = \int^{r}_{r'} (rh'(r) - h(r))' dr = \int^{r}_{r'} rh''(r) dr,
    \end{equation}
    and by the convexity of $H$, 
    \begin{equation} \label{eq: welldefinitinessofreparametrization}
       h'(r) - h'(r') = h'(r)|^{r}_{r'} \leq (rh'(r) -h(r))|^{r}_{r'} \leq r_{max}h'(r)|^{r}_{r'} = r_{max}(h'(r)- h(r')).
    \end{equation} 
    So, if  $h'(r)=h'(r')$, then from \eqref{eq: welldefinitinessofreparametrization}, we get $rh'(r) - h(r) = r'h'(r')-h(r')$, which proves that $\widetilde{A_h}$ is well defined. The monotonicity follows from \eqref{eq: welldefinitinessofreparametrization} since $h$ is convex, and continuity is a consequence of the formula. Now, since 
    $$\widetilde{A_h}(0) = h(0) = 0 \ \text{and} \ \widetilde{A_h}(T) = r_{max}h'(r_{max}) - h(r_{max}) = B,$$
    we conclude that $\widetilde{A_h}$ is a bijection between $[0,T]$ and $[0,B]$. 
\end{proof}
In what follows, we will be interested in the function $A_h = \widetilde{A_h} \circ (h')$, which we keep referring to as reparametrization function, from the interval $[1,r_{max}]$ to $[0,B]$. Since the convex semi-admissible Hamiltonians considered here are in particular strictly convex, we obtain that $A_h$ is, similarly to $\widetilde{A_h}$, a continuous monotone increasing bijection.  

\subsection{Floer homology and continuation maps} In this subsection, we establish notation and outline the construction of wrapped Floer homology.

An almost complex structure $J$ on $([1,\infty) \times \Sigma, \lambda =r \alpha)$ is called \emph{cylindrical} if
\begin{itemize}
    \item Preserves $\xi_{M}$, i.e., $J(\xi_{M})= \xi_{M}$,
    \item $J|_{\xi_M}$ is independent of $r$,
    \item $J(r \partial_{r}) = R_{\alpha}$, for $r\geq 1$.
\end{itemize}
The first and last conditions are equivalent to 
\begin{equation} \label{eq: relation between dr and alpha}
    dr\circ J = - r\alpha.
\end{equation}
If these conditions hold only for $r \geq r_{0}$, for some $r_{0}\geq 1$, then we call $J$ asymptotically cylindrical, or cylindrical at infinity. We say that an almost complex structure in $\widehat{M}$ \emph{admissible} if it is compatible with $\omega$ and cylindrical at infinity. 

Fix an admissible almost complex structure in $\widehat{M}$, and consider the $\omega$-compatible Riemannian metric $g$ in $\widehat{M}$ satisfying $\omega(\cdot, J \cdot) = g(\cdot, \cdot)$. By endowing $\widehat{\mathcal{P}}_{L_{0}\rightarrow L_{1}}$ with the induced metric from $g$, the solutions $u:\mathbb{R} \rightarrow \widehat{\mathcal{P}}_{L_0 \rightarrow L_1}$ of the gradient equation 
\begin{equation} \label{eq: gradient equation}
    \partial_{s}u =- \nabla \mathcal{A}_{H}(u),
\end{equation}
where $s\in \mathbb{R}$, for a Hamiltonian $H$, are equivalent to strips $u : \mathbb{R} \times [0,1] \rightarrow \widehat{M}$, refereed as \textit{Floer strips}, satisfying the equation
\begin{equation} \label{eq: Floer equation with boundary constraint}
    \begin{cases}
    $$\partial_{s}u -J(u)(\partial_{t}u - X_{H}(u))=0$$, \\
    $$u(\cdot , 0) \in \widehat{L}_{0}, \ u(\cdot , 1) \in \widehat{L}_{1}$$,
    \end{cases}
\end{equation}
where $(s,t) \in \mathbb{R}\times [0,1]$, which we will refer as \textit{Floer equation}.
Notice that the Floer strips decreases along $\mathcal{A}_H$, i.e., $s\rightarrow \mathcal{A}_H(u(s,\cdot))$ is decreasing.

We also point out that the leading term of this equation is not a $\Bar{\partial}$-operator, but a $\partial$-operator, that is, for $H \equiv 0$, the solutions for the resulting equation are anti-holomorphic curves instead of holomorphic curves. Regardless, there are no losses to work with the solutions of \eqref{eq: Floer equation with boundary constraint} instead of the classical Floer equation. Indeed, all the properties of the solutions of the classical Floer equation translates to this setting via the chance of variables $s\rightarrow -s$.

There are three facts about Floer strips we will use throughout this paper. In what follows, we fix an  admissible almost complex structure $J$, a convex (semi)-admissible Hamiltonian $H$, and $u$ a Floer strip asymptotic to the Hamiltonian chords $(r^{\pm},x^{\pm})$ as $s\rightarrow \mp \infty$, i.e., $u_s:= u(s,\cdot) \rightarrow (r^{\pm},x^{\pm})$ as $s\rightarrow \mp \infty$, where $x^{\pm}$ are Reeb chords in $\mathcal{T}_{\Sigma_0 \rightarrow \Sigma_1}(\alpha)$.

The first fact is the standard \textit{maximum principle}. The version that we will need states that for Floer strips $u$ as above do not go beyond the level $r^+$, i.e.,
\begin{equation} \label{eq: maximum principle}
    \sup_{(s,t) \in \mathbb{R}\times [0,\tau]} r(u(s,t)) \leq r^+.
\end{equation}
We emphasize that the version of the maximal principal we just described is not the most general one one may find. We refer to \cite{ritter2013topological} and references therein. 

The second fact we will use is less standard. We will refer to it as the Bourgeois-Oancea and it goes back to [\cite{bourgeois2009exact},p. 654]; see also [\cite{cieliebak2018symplectic}, Lemma 2.3]. It asserts that the curves $u_s$ need to raise at least to the level $r^-$, i.e.,

\begin{equation} \label{eq: Bourgeois-Oancea}
    \max_{t\in [0,\tau]} r(u(s,t)) \geq r^-,
\end{equation}
for all $s\in \mathbb{R}$. For the sake of completeness, we will present a proof of this fact on the last section. 

To state the third fact, we recall that the \textit{energy} of a Floer strip $u$ is defined by 

$$E(u):= \int_{\mathbb{R}\times [0,1]} ||\partial_s u(s,t)||^2dsdt.$$

When $u$ is asymptotic to the Hamiltonian chords $(r^{\pm},x^{\pm})$ as $s\rightarrow \mp \infty$, then 
$$E(u) = \mathcal{A}_h(r^+) - A_h(r^-).$$

Finally, the last fact we will use, which is also standard, states that for any $r>1$, there exists a constant $C=C(r,J,H)$ such that for any Floer strip $u$ as in \eqref{eq: Floer equation with boundary constraint} asymptotic to a chord $(r',x)$ at either end with $E(u)<C$, and $r'\geq r$, has image in $[1,\infty)\times \Sigma$, i.e.,

\begin{equation} \label{eq: curve confined in the cilyndrical part}
    u(s,t) \in [1,\infty)\times\ \Sigma, \text{ for all } (s,t) \in \mathbb{R}\times [0,1].
\end{equation}

\subsection{Filtered wrapped Floer homology}

In what follows, we summarize the construction of the wrapped Floer homology persistence module for (semi)-admissible Hamiltonians in Liouville domains, assuming the well definiteness of wrapped Floer homology for non degenerate, linear at infinity Hamiltonians. We refer to \cite{fernandes2024barcode} for a detailed construction of the filtered wrapped Floer homology with coefficients in $\mathbb{Z}/2\mathbb{Z}$ for linear at infinity Hamiltonians.

Let $H$ be a (semi)-admissible Hamiltonian with $slope(H) \notin \mathcal{S}(\alpha, \Lambda_0 \rightarrow \Lambda_1)$. For any $t \notin \mathcal{S}(H)$, we set
$$HW^t(H,L_0 \rightarrow L_1):= HW^t(\Tilde{H},L_0 \rightarrow L_1),$$
where $\Tilde{H}$ is a non degenerate perturbation of $H$ with $slope(\Tilde{H}) = slope(H)$. From the continuation map, we see that $HW^t(H,L_0 \rightarrow L_1)$ is independent of the perturbation $\Tilde{H}$, as long as the perturbation is sufficiently small. Note that, from the above definition, for $t',t \notin \mathcal{S}(H)$, with $t'<t$, we have the "inclusion map"
$$HW^{t'}(H,L_0 \rightarrow L_1) \rightarrow HW^t(H,L_0 \rightarrow L_1),$$
defined as the "inclusion map" of a small perturbation $\Tilde{H}$ of $H$.

Now, if $t \in \mathcal{S}(H)$, we set
$$HW^t(H,L_0 \rightarrow L_1) := \lim_{\begin{smallmatrix} \longrightarrow & \\ t' \leq t \end{smallmatrix}} HW^{t'}(H,L_0 \rightarrow L_1),$$
where the limit it taken with respect to the "inclusion maps". 

In order to remove the assumption $slope(H) \notin \mathcal{S}(H,L_0 \rightarrow L_1)$, we set
$$HW^t(H,L_0 \rightarrow L_1):= \lim_{\begin{smallmatrix} \longrightarrow & \\ H' \leq H \end{smallmatrix}}HW^{t}(H',L_0 \rightarrow L_1),$$
where the limit is taken over Hamiltonians $H'\leq H$ with $slope(H') \notin \mathcal{S}(\alpha,L_0 \rightarrow L_1)$. Note that 
$$HW^t(H,L_0 \rightarrow L_1) = \{0\} \text{ for } t\leq 0,$$
for convex semi-admissible Hamiltonians $H$. It turn out the family of vector spaces $t \rightarrow HW^t(H,L_0 \rightarrow L_1)$ together with the "inclusion maps" form a persistence module in the sense of Section 2 in \cite{fernandes2024barcode} for semi-admissible Hamiltonians $H$.

Finally, the filtered wrapped Floer homology $HW^t(M,L_0 \rightarrow L_1)$ is defined as

\begin{equation} \label{defwrap}
    HW^a(M,L_0 \rightarrow L_1) := \lim_{\begin{smallmatrix}
    \longrightarrow \\ H
\end{smallmatrix}} HW^a(H,L_0 \rightarrow L_1),
\end{equation}
where the limit is taken over the set of admissible Hamiltonians. Since convex admissible Hamiltonians form a cofinal sequence, we can consider $H$ on this class. Furthermore, we set

\begin{equation} \label{conventionwrapa<0}
    HW^a(M,L_0 \rightarrow L_1) := 0 \ \text{when} \ a \leq 0.
\end{equation}
This construction leads to a family of vector spaces which, together with some natural maps, constitute a persistence module in the sense of the definitions in section 2 on \cite{fernandes2024barcode}. More precisely:
\begin{proposition}[Corolary 1, \cite{fernandes2024barcode}]
    The family of vector spaces $t \rightarrow HW^t(M,L_0 \rightarrow L_1)$ is a persistence module, with structure maps the direct limit of the "inclusion" maps. 
\end{proposition}

\section{Wrapped Floer homology barcode entropy}
In this section, we revisit one of the central concepts of this paper: the wrapped Floer homology barcode entropy. Additionally, we introduce an equivalent definition that will be useful for the proof of Theorem B. 

For a Liouville domain $(M,\lambda)$, and two asymptotically conical Lagrangians $L_0$ and $L_1$, we define the \textit{wrapped Floer homology barcode entropy} in the following way:
\begin{definition}
    The wrapped Floer homology \textbf{barcode entropy} of $M$ is defined by 
    $$\hbar^{HW}(M,L_0 \rightarrow L_1) =\lim_{\epsilon \rightarrow 0} \limsup_{t \rightarrow \infty} \frac{\log^{+}\text{\textcrb}_{\epsilon}(M,L_0 \rightarrow L_1,t)}{t},$$
\end{definition}

We will drop the Lagrangians from the notation when they are well understood. 

In what follows, we present a more convenient way to compute the wrapped Floer homology barcode entropy using the filtered wrapped Floer homology of a single Hamiltonian. 

For any convex semi-admissible Hamiltonian $H$ in $\widehat{M}$ with $H(r,x)=rT-B$, for $r \geq r_{max}(H)$, and any $s>0$, we can consider the barcode associated to the persistence module $t \rightarrow HW^t(sH)$. We use this family of barcodes to define the \emph{wrapped Floer homology barcode entropy of $H$} as follows.

\begin{definition}
For a convex semi-admissible Hamiltonian $H$, we define the wrapped Floer homology barcode entropy of $H$ to be
$$\hbar^{HW}(H,L_0 \rightarrow L_1) =\lim_{\epsilon \rightarrow 0} \limsup_{s \rightarrow \infty} \frac{\log^{+}\text{\textcrb}_{\epsilon}(sH,sB)}{sT}.$$
\end{definition}

\begin{proposition}[Proposition 4, \cite{fernandes2024barcode}] \label{sequecialbarcodeentropy}
    For a convex semi-admissible Hamiltonian $H$ in a Liouville domain $(M,\lambda)$, and two exact asymptotically conical Lagrangians $L_0$ and $L_1$,  we have
    $$ \hbar ^{HM}(H,L_0 \rightarrow L_1) = \hbar ^{HM}(M,L_0 \rightarrow L_1).$$
\end{proposition}

\section{Proof of Theorem B}

This section is dedicated to the proof of Theorem \hyperref[thmB]{B}. We begin by recalling some key facts about hyperbolic dynamics and establishing a relative version of the equality between topological entropy and the exponential growth of periodic orbits in locally maximal hyperbolic sets, more precisely stated in Proposition \ref{prop: top ent coincide growth of per orb}. Next, we state a relative version of the \textit{Crossing Energy Theorem} \ref{thm: crossing energy theorem}, and use it to prove Theorem B. In the final part of this section, we present and prove a relative version of the \textit{Location Constraint Theorem} \ref{thm: location constraint}, which is not only essential to the proof of the Crossing Energy Theorem but also of independent interest. Lastly, we include a proof of the Bourgeois-Oancea result \eqref{eq: Bourgeois-Oancea} for the sake of completeness.

\subsection{Exponential growth of chords and topological entropy}
We start this subsection by fixing notation and recalling some facts about flows and hyperbolic sets. 

Consider $\varphi: \mathbb{R}\times M \rightarrow M$ is a $C^{\infty}$ flow in a closed manifold $M$. We denote $\varphi^t$ the time $t$ map of the flow, i.e., $\varphi^t(x) = \varphi(t,x)$. Let $K$ be a hyperbolic set for $\varphi$. For $q \in K$, and $\delta >0$, we denote by 
$$\stable{q} = \{y \in M \mid d(\varphi^t(y),\varphi^t(q)) \leq \delta, \text{ for all } t \geq 0\}$$
and
$$\unstable{q} = \{y \in M \mid d(\varphi^t(y),\varphi^t(q)) \leq \delta, \text{ for all } t \leq 0\},$$
the \textit{local stable} and \textit{unstable manifolds} of $q \in K$, respectively. For what follows it is worth noting that there exists $\delta = \delta^*(K)$ such that for all $\delta \leq\delta^*$, the local stable and unstable manifolds $\stable{q}$ and $\unstable{q}$ are embedded disks in $M$, and there exist $\lambda,c>0$ such that 
$$\varphi^t(\stable{q}) \subset W^s_{\delta c e^{-\lambda t}}(\varphi^t(q)), \text{ for all } t\geq 0,$$
and
$$\varphi^t(\unstable{q}) \subset W^u_{\delta c e^{\lambda t}}(\varphi^t(q)), \text{ for all } t \leq 0.$$
Furthermore, $\delta^*$ is an \textit{expansivity constant}, i.e., for $x,y \in K$, $s:\mathbb{R} \rightarrow \mathbb{R}$, $s(0)=0$ with $d(\varphi^t(x),\varphi^{s(t)}(y)) \leq \delta^*$, for all $t \in \mathbb{R}$, then $\varphi^t(x) = y$, for some $t \in \mathbb{R}$. We point out that, in the literature, the expansivity of a flow has a stronger definition than the one presented above. However, since we are only interested in distinguishing orbits, the description provided here is sufficient.

A hyperbolic set $K$ for a flow $\varphi^t$ is called \textit{locally maximal} if there exists a neighborhood $K \subset U$ such that $K$ is the largest invariant set in  $U$ by the flow, i.e., 
$$\bigcap_{t \in \mathbb{R}} \varphi^t(U) = K.$$

For $x \in M$, $\mathcal{O}(x)=\{\varphi^t(x) \mid t \in \mathbb{R}\}$ denotes the \textit{full orbit} of $x$, and $\mathcal{S}^{\tau}(x)= \{\varphi^t(x) \mid 0 \leq t \leq \tau \}$ the \textit{orbit segment} of length $\tau$ starting at $x$. For $q,p \in K$, $\tau \geq 0$, and $\delta \geq 0$, we say that $\mathcal{S}^{\tau}(x)$ is a $(q,p,\tau,\delta)$-chord if $\mathcal{S}^{\tau}(x) \subset K$, $x\in W^u_{\delta}(q)$, and $\varphi^{\tau}(x) \in W^s_{\delta}(p)$. We consider two $(q,p,\delta)$-chords to be the same if their full orbits are the same. 

In what follows, $C(q,p,\tau,\delta)$ denotes the set of $(q,p,\delta)$-chords of length at most $\tau$, and  $N(q,p,\tau,\delta)$ is the number of elements in $C(q,p,\tau,\delta)$. The next proposition is a relative version of the well known fact that, on a locally maximal hyperbolic set $K$ for $\varphi$, the topological entropy $\tope{\varphi|_{K}}$ coincides with the exponential growth of periodic orbit in $K$ with respect to the period. 


\begin{proposition} \label{prop: top ent coincide growth of per orb}
Let $K$ a locally maximal, topologically transitive hyperbolic set for a flow $\varphi$. Then for any $q,p \in K$, $0<\delta < \delta^* := \delta^*(K)/2$,
$$\limsup_{\tau \rightarrow \infty} \log \frac{N(q,p,\tau,\delta)}{\tau} = \tope{\varphi|_{K}}.$$
\end{proposition}
\begin{proof}
    We start by showing that 
    \begin{equation} \label{eq: inq exp growth nbm per orb bound by ent}
        \limsup_{\tau \rightarrow \infty} \log \frac{N(q,p,\delta,\tau)}{\tau} \leq \tope{\varphi|_K}.
    \end{equation}
    Consider two distinct chords $\mathcal{S}^{t_1}(x)$ and $\mathcal{S}^{t_2}(y)$ in $C(q,p,\tau,\delta)$. Then there exists $0 \leq t \leq \tau$ such that $d(\varphi^t(x), \varphi^t(y)) \geq \delta$, otherwise $d(\varphi^t(x),\varphi^t(y)) \leq \delta^*$ for all $t \in \mathbb{R}$ and by the expansivity property of $\delta^*$, $\mathcal{O}(x) = \mathcal{O}(y)$. By the formulation of $\tope{\varphi|_K}$ using the exponential growth of the number of separated sets, it is easy to see that \eqref{eq: inq exp growth nbm per orb bound by ent} holds.
    
    For the reserve inequality, i.e., 
    \begin{equation} \label{eq: inq exp growth nbm per orb bound by ent 2}
        \limsup_{\tau \rightarrow \infty} \log \frac{N(q,p,\delta,\tau)}{\tau} \geq \tope{\varphi|_K},
    \end{equation}
    we make use of the \textit{Specification Theorem}. Let us begin by defining what a specification is. A \textit{specification} $S=(\eta,P)$ for $\varphi$ is a pair consisting of a finite collection $\eta = \{I_1,...,I_m\}$ of bounded intervals $I_i = [a_i,b_i] \subset \mathbb{R}$, and a map $P: T(\eta): \cup_{I \in \eta} I \rightarrow X$ such that for $t_1,t_2 \in I \in \eta$, we have $\varphi^{t_2 - t_1}(P(t_1)) = P(t_2)$. $S$ is said to be $M$-\textit{spaced} if $a_{i+1} > b_i + M$ for all $i\in \{1,...,m\}$, and the minimal such $M$ is called the \textit{spacing} of $S$. We say that $S$ is $\epsilon$-\textit{shadowed} by $x \in M$ if $d(\varphi^t(x),P(t))<\epsilon$ for all $t\in T(\eta)$. We say that $\varphi$ has the \textit{specification property} if for any $\epsilon>0$ there exists an $M=M_{\epsilon}>0$ such that any $M$-spaced specification $S$ is $\epsilon$-shadowed by a point $x\in M$. The Specification Theorem [\cite{fisher2019hyperbolic}, Theorem 5.3.61] says that if $K$ is a topologically transitive compact locally maximal hyperbolic set for a flow $\varphi$, then the restriction flow $\varphi|_{K}$ has the specification property. 
    
    For any element $y$ of a $(\tau,\epsilon)$-separated set $E_{\tau}$, we consider the specification $S=(\{I_1,I_2,I_3\}, P)$, where $I_1=\{0\}$, $I_2 = [M_{\epsilon/2},\tau + M_{\epsilon/2}]$ and $I_3=\{\tau + 2M_{\epsilon/2}\}$, and $P(0) = q$, $P(t) = \varphi^{t - M_{\epsilon/2}}(y)$ for $t\in I_2$ and $P(\tau + 2M_{\epsilon/2}) = p$. Then by the Specification Theorem, there exists a point $x\in K$ that $\epsilon$-shadows $S$. From the proof of the Specification Theorem, we obtain $x$ in such a way that $\mathcal{S}^{\tau + 2M_{\epsilon/2}}(x)$ is a $(q,p,\tau +2M_{\epsilon/2})$-chord. Therefore we conclude that $N(q,p,\tau + 2M_{\epsilon/2},\delta) \geq |E_n|$, and again by the formulation of $\tope{\varphi|_K}$ using the the exponential growth of the number of separated sets, it is easy to see that \eqref{eq: inq exp growth nbm per orb bound by ent 2} holds. 
\end{proof}

\subsection{Crossing Energy and proof of Theorem B} In this subsection, we state the Crossing Energy Theorem with Boundary Conditions and provide a proof of Theorem B. From now on, we adopt the convention that all considered Hamiltonians are convex semi-admissible unless otherwise stated. 

For the next theorem, we fix asymptotically conical Lagrangians $L_0$ and $L_1$ in $M$, with $ \partial L_0 =\Lambda_0$ and $\partial L_1 = \Lambda_1$ Legendrians in $\Sigma$.
\begin{theorem}[Crossing Energy Theorem with Boundary Conditions] \label{thm: crossing energy theorem}
Let $K$ be a locally maximal, topologically transitive compact hyperbolic set of $\alpha$. Assume that $W^u_{\delta}(q) \subset \Lambda_0$ and $W^s_{\delta}(p) \subset \Lambda_1$, for some $q,p \in K$, and $\delta>0$. Fix an interval 
$I = [r_-, r_+] \subset (1,r_{max})$,
and let $H(r,x)= h(r)$ be a semi-admissible Hamiltonian with $slope(H):= a \notin \mathcal{S}(\alpha)$ such that
$$h^{'''}\geq 0 \text{ on } [1,r_+ + \delta],$$
for some $\delta$ with $r_+ + \delta < r_{max}$. Fix an admissible almost complex structure $J$. Furthermore, let $z$ be a $(q,p,\delta)$-chord of $\alpha$ with length $T$, so that the corresponding $1$-chord $\Tilde{z} = (r^*,z)$ of $\tau H$ has $r^* \in I$ (Hence $\tau a \geq T$ and $r^*$ depends on $\tau$).

Then there exists $\sigma>0$ such that $E(u) \geq \sigma$, independent of $\tau$ and $z$, for any non-trivial $(E(u)\neq 0)$ Floer strip $u: \mathbb{R}\times [0,1] \rightarrow \widehat{M}$ of $\tau H$ asymptotic, at either end to $\Tilde{z}$.
\end{theorem}
\begin{remark} \label{rmk:lower bound on the energy for a perturbation}
    Fix $0<\sigma' < \sigma$ and $\tau \geq 0$. For a non-degenerate $C^{\infty}$ small perturbations $\Tilde{H}$ of $\tau H$, the $\tau$-chords $\Tilde{z}$ of $H$ from Theorem \ref{thm: crossing energy theorem} split into non-degenerate $1$-chords of $\Tilde{H}$ in a tubular neighborhood of $\Tilde{z}$. Then it follows from a suitable version of Gromov compactness that any Floer strip of $(\Tilde{H},L_0 \rightarrow L_1)$ asymptotic to theses chords at either end has energy at least $\sigma'$.
\end{remark}
The following proposition is a technical result in persistence modules that is crucial to the proof of Theorem B. The first version of this result was proved in \cite{cineli2021topological}, but the version we use here can be found in \cite{meiwes2018rabinowitz}. 
\begin{proposition}[Proposition 4.4, \cite{meiwes2024barcode}] \label{prop: bar in a isolated set}
Assume that $(C'_1,\partial_1 '),...,(C'_p,\partial_p'))$ are chain complexes with non-zero homology that are $\epsilon$-isolated in $(C,\partial)$, and for which $C_i' \cap C_j' = \{0\}$ if $i\neq j$. Then $b_{\epsilon}(C)\geq p/2$.
    
\end{proposition}
\begin{proof}[Proof of Theorem B] Throughout the proof we fix $\delta<\delta^*/2$. Now Let $p(s) = N(q,p,\delta,s)$. From Proposition \ref{prop: top ent coincide growth of per orb}, we have
    $$L:= \tope{K} = \limsup_{s\rightarrow \infty} \frac{log^+p(s)}{s},$$
    since $K$ is a locally maximal, topologically transitive hyperbolic set. Fix $r_- \in (1,r^+)$, and set $I' = (r_-,r_+]$. Notice that in Theorem \ref{thm: crossing energy theorem}, $I = [r_-,r_+]$, and hence $I' \subset I$. 
    
    Let $p^H(s)$ be the number of $s$-chords $\Tilde{z}=(r',z)$ for $(H,L_0 \rightarrow L_1)$ with $z \in C(q,p,\delta)$ and $r'\in I'$. Equivalently, $p^H(s)$ is the number of $1$-chords $\Tilde{z}$ of $s H$ with $z \in C(q,p,\delta)$ and $r'\in I'$.

    We want to show that 
    \begin{equation} \label{eq:inequality on the growth of orbits}
        h'(r_+)\cdot L \leq \limsup_{s\rightarrow \infty} \frac{log^+p^H(s)}{s} \leq L.
    \end{equation}
    
    For the second inequality, notice that $p^H(s) \leq p(s)$, since there is a a one-to-one correspondence between chords $z\in C(q,p,\delta,s)$ (assume $slope(H)=1$) and the $s$-chords $\Tilde{z}=(r,z)$ of $(H,L_0\rightarrow L_1)$ with $z\in C(q,p,\delta,s)$.
    
    We now discuss the first inequality. Let $\Tilde{z} = (r^*,z)$ be an $\tau$-chord of $(H,L_0 \rightarrow L_1)$ and $T$ the length of the chord $z$ in $C(q,p,\delta,s)$. Then, $T$ and $r^*$ are related by the condition 
    $$sh^{'}(r^*) = T.$$
    Set 
    $$a_- := h^{'}(r_-) \text{ and } a_+:=h^{'}(r_+).$$
    Then, for $r^*$ to be in $I'=(r_-,r_+]$ we must have 
    $$sa_- < T \leq sa_+,$$
    and thus 
    \begin{equation} \label{eq:p^h using the interval bounds}
        p^H(s) = p(sa_+) - p(sa_-).
    \end{equation}
    (The reason that we took $I'$ to be an semi-open interval rather than the closed interval $I$ is to ensure that this equality holds literally). 
    We have
    $$p(sa_+) = (p(sa_+) - p(sa_-)) + p(sa_-) \\
     \leq \max\{2(p(sa_+) - p(sa_-)),2p(sa_-)\}.$$
     Therefore,
     \begin{align*}
         a_+\cdot L & = \limsup_{s\rightarrow \infty} \frac{\log^+ p(sa_+)}{s} \\
         & \leq \limsup_{s\rightarrow \infty} \frac{\log^+ \max\{2(p(sa_+) - p(sa_-)),2p(sa_-)\}}{s} \\
         & \leq \limsup_{s\rightarrow \infty} \frac{\max\{ \log^+(2(p(sa_+) - p(sa_-))), \log ^+2p(sa_-) \}}{s} \\
        & \leq \max\left\{ \limsup_{s\rightarrow \infty} \frac{\log^+(2(p(sa_+) - p(sa_-)))}{s}, \limsup_{s\rightarrow \infty} \frac{\log^+2p(sa_-)}{s} \right\} \\
         & \leq \max\left\{ \limsup_{s\rightarrow \infty} \frac{\log^+p^H(s)}{s}, a_- \cdot L\right\}.
     \end{align*}
     Here, on the last equality we used \eqref{eq:p^h using the interval bounds}. The second term in the last line is strictly smaller than $a_+L$ since $a_- < a_+$, and hence the first term must be greater than or equal to $a_+ L$. This proves the first inequality in \eqref{eq:inequality on the growth of orbits}.
     
     Now, from Remark \ref{rmk:lower bound on the energy for a perturbation}, for a $C^{\infty}$-small perturbation $\Tilde{H}$ of $sH$, the $s$-chords $\Tilde{z} = (r^*,z)$ of $(H,L_0 \rightarrow L_1)$ with $r^* \in I'$ split into non-degenerate $1$-chords of $(\Tilde{H},L_0 \rightarrow L_1)$, 
     and the Floer strips asymptotic to these orbits at either end have energy at least a chosen $\sigma'$, with $0< \sigma ' < \sigma$, independent of $s$ and $\Tilde{z}$ (the size of the perturbation $||sH -\Tilde{H}||_{C^{\infty}}$ may depend on $s$). Note that $sA_h(r_+) < sB - \epsilon$ for $s$ sufficiently large, thus the action of all such chords for $(\Tilde{H},L_0 \rightarrow L_1)$   is below $sB-\epsilon$.

   
     Now let $2\epsilon< \sigma'$. Then, from Proposition \ref{prop: bar in a isolated set}, it follows that the persistence module $\tau \rightarrow WF^{\tau}(\Tilde{H})$ has at least $p^H(s)/2$ bar of length at least $2\epsilon$, i.e., $\text{\textcrb}_{2\epsilon}(\Tilde{H},sB -\epsilon) \geq p^H(s)/2$. Since $||sH - \Tilde{H}||_{C^{\infty}}$ is sufficiently small, then from Proposition 2  in \cite{fernandes2024barcode}, we have that $\text{\textcrb}_{\epsilon}(sH,sB - \epsilon) \geq p^H(s)/2$ (we choose the perturbation $\Tilde{H}$ after $\epsilon$ is fixed). Therefore,
     from Proposition 4 in \cite{fernandes2024barcode} and \eqref{eq:inequality on the growth of orbits}, we obtain
     \begin{align*}
          \be{M,L_0 \rightarrow L_1} & \geq \limsup_{s \rightarrow \infty} \frac{\log^+ \text{\textcrb}_{\epsilon}(sH,sB - \epsilon)}{s} \\
         & \geq \limsup_{s\rightarrow \infty} \frac{\log^+ p^H(s)}{s} \\
         & \geq h'(r_+) \cdot \tope{K} \\
         & \geq (1-\eta) \cdot \tope{K}.
     \end{align*}
     Since we can choose $\eta$ arbitrarily small, we conclude \ref{eq:barcode entropy bounded by below}.
    
\end{proof}

\subsection{Location Constraint and proof of Crossing Energy Theorem} 


Throughout this subsection will be convenient to adopt a different point of view for the filtered wrapped Floer homology of $\tau H$, where $H$ is a Hamiltonian with $slope(H)=1$. Instead of looking at the $1$-chords of $\tau H$, we will consider the $\tau$-chords of $H$. This modification does not affect the ingredients that go into the set up of the wrapped Floer homology groups, i.e., it does not affect the Floer complex, the action and action filtration, the energy of Floer strips and so on. In this setting, we will refer to the filtered Floer homology of $\tau H$ as the filtered Floer homology of $H^{\sharp \tau}$. Furthermore, we will assume that all the Floer strips have energy small enough so that their image is contained in $[1,\infty)\times \Sigma$, as in \eqref{eq: curve confined in the cilyndrical part} (all the Floer strips considered here are asymptotic at either end to a chord that lives in a level $r\geq r_0>1$, for some fixed $r_0$).

\begin{remark} \label{rmk: partial s u bounded by energy}
    From the work in \cite{salamon1990morse} and [\cite{salamon1999lectures}, sec 1.5], there exist constants $C_H>0$ and $\sigma_H>0$ such that for any Floer strip $u:\mathbb{R}\times [0,\tau] \rightarrow \widehat{M}$ for $H^{\sharp \tau}$ with $E(u)<\sigma_H$,
\begin{equation} \label{eq: curve pointwisely bounded by energy}
    ||\partial_s u(s,t)|| < C_H \cdot E(u)^{1/4},
\end{equation}
where the $(s,t) \in \mathbb{R}\times [0,\tau]$ is the coordinates, and the norm on the left is the $L^{\infty}$; see also \cite{bramham2015pseudo}. The constants $\sigma_H$ and $C_H$ depend on $J$, and the first and second derivatives of $H$, but not on $\tau$ or $u$. This is one instance where it is convenient to work with $H^{\sharp \tau}$ instead of $\tau H$.  
\end{remark}
The following theorem is not only essential for proving the Crossing Energy Theorem but is also of independent interest. It can be viewed as a relative version of Theorem 6.1 in \cite{cineli2023invariant}.
\begin{theorem}[Location Constraint Theorem with Boundary Conditions] \label{thm: location constraint}
Let $H(r,x) = h(r)$ be a semi-admissible Hamiltonian. Assume that $1<r_*^- \leq r_*^+$ and $\delta>0$ are such that 
$$1<r_*^- - \delta \text{ and } r_*^+ + \delta < r_{max},$$
and
$$h^{'''}\geq 0 \text{ on } [1,r_*^+ + \delta].$$  
Fix an admissible almost complex structure $J$. There exists $\sigma_0 > 0$ such that for any Floer strip $u: \mathbb{R} \times [0,\tau] \rightarrow \widehat{M}$ satisfying Floer's equation \eqref{eq: Floer equation with boundary constraint} for $(H^{\sharp \tau}, L_0 \rightarrow L_1)$ and any $\tau>0$, with energy $E(u)<\sigma_0$ and asymptotic to Hamiltonian chords at both ends, one of which, at either end, is in $[r_*^- , r_*^+] \times \Sigma$, has image contained in $(r_*^- - \delta, r_*^+ + \delta)\times \Sigma$.
\end{theorem}

Theorem \ref{thm: location constraint} is a consequence of the following result:

\begin{theorem} \label{thm: lower bound for the r coordinate of a curve}
    Let $H(r,x) =h(r)$ be a semi-admissible Hamiltonian. Fix an admissible almost complex structure $J$. Then there exist $\epsilon>0$ and $C>0$ such that for any $\tau \geq 1$, and any $u:\mathbb{R}\times [0,\tau] \rightarrow \widehat{M}$ Floer strip for ($H^{\sharp \tau}, L_0 \rightarrow L_1$) with $E(u)<\epsilon$, contained in a region where $h''' \geq 0$ and $r^{\pm} = r(u(\pm \infty,t))$, we have
    \begin{equation} \label{eq:lower_bound_for_r_coordinate}
        \inf_{\mathbb{R} \times [0,\tau]} r(u(s,t)) \geq r^- - \frac{C\cdot(r^+)^{3/4} E(u)^{5/8}}{\sqrt{A_h(r^+)}},
    \end{equation}
    where $A_h(r^+) = r^+ h'(r^+) - h(r^+)$.
\end{theorem}

\begin{proof}[Proof of Theorem \ref{thm: location constraint}]
    Let $\epsilon>0$ as in Theorem \ref{thm: lower bound for the r coordinate of a curve}. Choose $0<\sigma' < \epsilon$ such that $r^+ - r^- < \delta/2$, for any $\tau > 0$, and any Floer strip $u: \mathbb{R} \times [0,\tau] \rightarrow \widehat{M}$ of $H^{\sharp \tau}$ with $E(u)< \sigma '$, which is asymptotic to a chord in $[r_*^-, r_*^+] \times \Sigma$ at at least one of the ends $u(\pm \infty, t)$. Such choice of $\sigma '$ exists since $A_h$ is continuous and increasing ($A'_h(r)= rh''(r)>0$ in $(1,r_{max})$). Now, we take $\sigma < \sigma'$ satisfying 
    $$\sigma^{5/8} < \frac{\delta \sqrt{A_h(r^-_*)}}{2C(r_*^+)^{3/4}},$$
    or equivalently 
    $$\frac{C(r^+_*)^{3/4}\sigma^{5/8}}{\sqrt{A_h(r_*^-)}}< \delta/2,$$
    and apply Theorem \ref{thm: lower bound for the r coordinate of a curve} to $E(u)< \sigma$. Then it follows from \eqref{eq:lower_bound_for_r_coordinate} and the maximum principle \eqref{eq: maximum principle} that the image of $u$ is contained in $(r_*^- - \delta,r_*^+ + \delta)\times \Sigma$
\end{proof}

The proof of Theorem \ref{thm: lower bound for the r coordinate of a curve} is highly technical and not particularly enlightening, so we defer it to the end of this section. In what follows, we present the proof of the Crossing Energy Theorem.

\begin{proof}[Proof of Theorem \ref{thm: crossing energy theorem}] Throughout the proof, we assume $E(u)<\sigma_H$, so $u$ satisfies \eqref{eq: curve pointwisely bounded by energy}. \\

\textbf{Step 1:} A maps $\gamma$ from a interval is called a \textit{$\eta$-pseudo-orbit} or sometimes refereed just as \textit{pseudo-orbit} for the flow of a vector field $X$ if 
$$||\dot{\gamma}(t) - X(\gamma(t))||<\eta,$$
for all $t$ in the domain of $\gamma$. Then, if $\eta$ is small enough, $\gamma$ is close to the integral curve o the flow starting at $x=\gamma(0)$. More precisely, whenever the closed interval $\mathcal{I}$ in the domain of $\gamma$ is fixed, and $\eta$ is sufficiently small, $\gamma|_{\mathcal{I}}$ is point-wise close to the integral curve of the flow starting at $x$ on the same interval $\mathcal{I}$.

From Inequality $\eqref{eq: curve pointwisely bounded by energy}$ and the Floer equation \eqref{eq: Floer equation with boundary constraint}, the curves $u_s$ given by
$$u_s(t) = u(s,t),$$
are $\eta$-pseudo-orbits of $X_H$ on the interval $[0,\tau]$, with $\eta = C_H\cdot E(u)^{1/4}$. Thus, when $E(u)$ is small enough, each $u_s$ approximates the integral curve of $X_H$ on the interval $[0,\tau]$.

On this step we would like to say that the projections of $u_s$ to $\Sigma$ are $\eta$-pseudo-orbits for the Reeb flow, with a potentially different (but still controlled) $\eta$.

Let $u$ be a Floer strip for $H^{\sharp \tau}$ with $E(u)<\sigma_H$, and assume that $\inf r(u) \geq r_{min}$, with $r_{min}>1$ given by Theorem \ref{thm: location constraint}. Notice that from the maximal principal, the image of $u$ is contained in $\Sigma\times [r_{min},r_{max}]$. Set $\epsilon = E(u)^{1/4}$, and $C=C_H$. Then, from \eqref{eq: curve pointwisely bounded by energy}, we have that
$$||\partial_t u - X_H(u)||< C \epsilon,$$
for all $s\in \mathbb{R}$. Since $X_H = h'R_{\alpha}$ on each level $\{r\}\times \Sigma$, for $r \geq 1$, by denoting by $v$ the projection of $u$ to $\Sigma$, we obtain that
\begin{equation}
    \label{eq: projections as pseudo reeb orbits}
    ||\partial_t v - h'(r(u))R_{\alpha}(v)|| < C \epsilon.
\end{equation}

For each $s \in \mathbb{R}$, we reparameterize the map $u_s$ on the interval $[0,\tau]$ by using the change of variable $t=t(\xi)$ satisfying $\xi(0)=0$ and
$$t'(\xi) = \frac{1}{h'(r(u(s,t)))}>0,$$
and set 
$$\gamma_s(\xi) = u_s(t(\xi))$$
defined on the interval $[0,\tau_s]$, where 
$$\tau_s = \int_0^{\tau} h'(r(s,t))dt.$$
Now, we notice that 
$$||\dot{\gamma}_s(\xi) - R_{\alpha}(\gamma_s(\xi))||<\frac{C\epsilon}{\min\limits_{t\in[0,\tau]}h'(r(s,t))} \leq \frac{C\epsilon}{h'(r_{min})}:=\eta,$$
where the dot represents the derivative with respect to $\xi$. Therefore, we conclude that $\gamma_s$ are $\eta$-pseudo-orbits for the Reeb flow, where $\eta$ can be taken depending just on $H$ but independently of $s$ and $u$ (as long as the energy of $u$ is sufficiently small). \\

\textbf{step 3:} Let $K$ a locally maximal, topologically transitive compact hyperbolic set for the Reeb flow. Consider $U'=U(K)$ a isolating neighborhood for $K$ given by the Shadowing Theorem for flows [\cite{katok1995introduction}, Theorem 18.1.7]. Fix $\delta>0$ such that 
\begin{enumerate}
    \item $U = N_{\delta} := \{y \in M\ \mid \inf\limits_{x\in K}d(y,x) \leq \delta\} \subset U'$,
    \item $\delta < \delta^*/2$, where $\delta^*$ is the expansive constant for $K$,
    \item $\delta < \delta '/2$, where $\delta' = d(\partial U',K)>0$.
\end{enumerate}

Then, from the Shadowing Theorem for flows, for $\delta>0$, there exists an $\eta>0$ such that any continuous family of $\eta$-pseudo-orbits $\mathcal{Q}_s$ in $U$, with $s\in \mathbb{R}$, is $\delta$-shadowed by a continuous family of orbits $(\mathcal{O}(x^s))_{s\in \mathbb{R}}$. Moreover, by our choice of $\delta$ and the local maximality of $K$, we have that the orbits $\mathcal{O}(x^s)$ are in $K$, for all $s \in \mathbb{R}$, and that any $\eta$-pseudo-orbit  that intersects $\partial U$ is not entirely contained in $\overline{U}$.

We point out that to finish the proof, it is sufficient to show that there exists an $0<\epsilon<\sigma_H$ such that for any $u \in \mathcal{M}(\hat{x}_-, \hat{x}_+, H^{\sharp \tau},J)$, with $\hat{x}_- = (r^-,x_-)$ and $\hat{x}_+ = (r^+,x_+)$, where $x_-$ or $x_+$ are $(q,p,\delta)$-chords in $K\times I$, and $E(u)<\epsilon$, then $E(u) = 0$. Without loss of generality, we assume that $x^+$ is a $(q,p,\delta)$-chord and $r^+ \in I$. Set $y^s:= v(s,0) \in \Lambda_0$, and consider the family 
\begin{center}
      $y_s(\xi) :=
    \begin{cases}
        \gamma_s(\xi) & \text{for $\xi \in [0,\tau_s]$}, \\
        \varphi^{\xi}_{\alpha}(y^s) & \text{for $\xi<0$}, \\
        \varphi^{\xi - \tau_s}(y_s(\tau_s)), & \text{for $\xi > \tau_s $}.
    \end{cases}$
\end{center}
From step 2, $\mathcal{Q}_s =\{y_s\}_{s\in \mathbb{R}}$ is a continuous family of pseudo orbits for the Reeb flow, with $y_s  \rightarrow x_+$. Since $x_+$ is a $(q,p,\delta)$-orbit, then we conclude that $y_s(0) \rightarrow q \in \Lambda_0$, and $y_s(\tau_s) \rightarrow p \in \Lambda_1$. Moreover, $y_s(0 ) \in W^u_{\delta}(q)$ and $y_s(\tau_s) \in W^s_{\delta}(p)$ for $s$ sufficiently large. From our choice of $\delta$, it follows that $y_s$ is in $U$ for all $s$ sufficiently large. Indeed $y_s$ are in $U$ for all $s \in \mathbb{R}$, otherwise, by taking $s_0 = \inf\{s \in \mathbb{R} \mid y_s \subset U\}$, we conclude that $y_{s_0}$ intersects $\partial U$, which contradicts our choice of $\delta$ since $y_{s_0}$ would need to escape $\overline{U}$. In particular, we conclude that $x_-$ is in $K$, since $y_s \rightarrow x_-$, as $s \rightarrow -\infty$. Now, from the Shadowing Theorem, we obtain a family $\mathcal{O}(x_s)_{s\in \mathbb{R}}$ of orbits of the Reeb flow $K$ that $\delta$-shadows the pseudo orbits $\mathcal{Q}_s$. Moreover, as it is easy to verify, the orbit segments $S^{\Tilde{\tau}_s}(x_s)$ form a family of $(q,p,2\delta)$-chords for $s$ sufficiently large. Since $2\delta$ is an expansivity constant, we get that $x_s = x_+$ and $x_{-s} = x_-$, for all $s$ sufficiently large. Finally, since $\varphi(W^u_{\delta}(q))$ intersects $\varphi(W^s_{\delta}(p))$ transversely, we conclude that the $(q,p,\delta)$-chords are isolated and therefore $x_-=x_+$, which implies that $E(u) = 0$. 
\end{proof}

\begin{proof}[Proof of Theorem \ref{thm: lower bound for the r coordinate of a curve}] The proof will be divided in four steps. \\

\textbf{Step 1:} Fix a admissible Hamiltonian $H(r,x)=h(r)$, and let $u$ a Floer strip for $H^{\sharp \tau}$. Denote by $u_s : = u(s,.): [0,\tau] \rightarrow \widehat{M}$ the $s$-slice of the strip $u$. On this step, we want to show that 
\begin{equation}\label{eq:variation of r(u_s) is bounded by above}
\frac{d}{ds}\int_0^{\tau}r(u_s)dt \leq -\tau A_h(r^-) + \int_0^{\tau} A_h(r(u_s))dt.
\end{equation}
Notice that from \eqref{eq: relation between dr and alpha} and \eqref{eq: gradient equation}, we have
\begin{align*}
    \partial_s(r(u)) & = dr(\partial_s u) \\
    & = dr(J(\partial_t u - X_H(u))) \\ 
    & = -r(u)\alpha(\partial_t u ) + r(u)\alpha(X_H(u)) \\
    & = -r(u)\alpha(\partial_t u) + r(u)h'(r(u)) \\ 
    & = -[r(u)\alpha(\partial_t u) - h(r(u))] + [r(u)h'(r(u)) - h(r(u))].
\end{align*}
Integrating over the interval $[0,\tau]$, we obtain 
$$\frac{d}{ds} \int_0^{\tau} r(u_s) dt = - \mathcal{A}_{H^{\sharp \tau}}(u_s) + \int_0^{\tau} A_h (r(u_s))dt,$$
and since from \eqref{eq: gradient equation} the action decreases along Floer strips, we conclude that 

$$\mathcal{A}_{H^{\sharp \tau}}(u_s) \geq \mathcal{A}_{H^{\sharp \tau}} (u_{\infty}) = \tau A_h(r^-)$$
and inequality \eqref{eq:variation of r(u_s) is bounded by above} follows. \\

\textbf{Step 2:} From now on, we assume that the image of $u$ is contained in a region where $h''' \geq 0$. On this step we want to show that 
\begin{equation} \label{eq: step 2}
    \frac{A_h(r^-)}{A_h(r^+)}\tau r^+ \leq \inf_{s\in \mathbb{R}} \int_0^{\tau} r(u_s)dt \leq \tau r^+.
\end{equation}
\begin{remark} \label{rmk:rmk inside constraint lemma}
    Note that $E(u) = \tau (A_h(r^+) - A_h(r^-))$, and that the lower bound on \eqref{eq: step 2} can be written as $\tau r^+ - E(u)r^+/ A_h(r^+)$.
\end{remark}

Notice that the last inequality follows from the maximum principal \eqref{eq: maximum principle}. Now, for the first inequality we use the following claim: \\ \\
\textit{Claim 1:} The function $\frac{A_h(r)}{r}$ is non decreasing in $[1,r_0]$, provided that $h''' \geq 0$ on this interval. \\

Assuming the claim, let us prove the first inequality in \eqref{eq: step 2}, i.e.,

\begin{equation} \label{eq: small ineq on step 2}
    \frac{A_h(r^-)}{A_h(r^+)}\tau r^+ \leq  \int_0^{\tau} r(u_s)dt, \text{ for all } s\in [-\infty,\infty].
\end{equation}
First, note that from the claim and that $A_h(r^-) \leq A_h(r^+)$, follows that
$$\frac{A_h(r^-)}{A_h(r^+)} \tau r^+ \leq \tau r^-,$$
and 
$$\frac{A_h(r^-)}{A_h(r^+)} \tau r^+ \leq \tau r^+,$$
respectively. Hence, inequality \eqref{eq: small ineq on step 2} holds for $s = \pm \infty$. To conclude \eqref{eq: small ineq on step 2} for all $s \in [-\infty,\infty]$, it is enough to check that the inequality hold on all the critical points of the right hand side, i.e., on all the critical points of the function
$$s \rightarrow \int_0^{\tau} r(u_s)dt.$$
Thus, let $s_0$ be a critical point. From \eqref{eq:variation of r(u_s) is bounded by above}, the maximum principal \eqref{eq: maximum principle} and the claim above it follows that 
$$\tau A_h(r^-) \leq \int_0^{\tau} A_h(r(u_{s_0})) \leq \frac{A_h(r^+)}{r^+} \int_0^{\tau}r(u_{s_0}).$$
This completes the proof of \eqref{eq: small ineq on step 2} assuming the claim. 
\begin{proof}[Proof of Claim 1] 
    Recall that $A_h = rh' - h$. Dividing by $r$ and differentiating we obtain 
    $$\frac{d}{dr} \frac{A_h(r)}{r} = \frac{rA_h' - A_h}{r^2} = \frac{r^2h'' - rh' +h}{r^2}.$$
    Thus, it is enough to check that $f = r^2 h'' -rh' + h \geq 0$. Since $h''' \geq 0$, then $h''$ in monotone increasing. Now notice that 
    $$h'(r) - h'(1) = \int_1^r h''(h)dh \leq \int_1^r h''(r)dh \leq (r-1)h''(r),$$
    and since $h'(1)=0$, we obtain 
    $$h'(r) \leq (r-1)h''(r).$$
    Therefore,
    $$f = r(rh'' - h') + h \geq r(rh'' - (r-1)h'') + h \geq rh'' + h \geq 0.$$
    With this argument, we conclude that $\frac{A_h (r)}{r}$ is strictly increasing unless $h \equiv 0$, which is ruled out on this setting by our choice of Hamiltonian ($H$ is in particular strictly convex).
    \end{proof}
    \textbf{Step 3:} For $s\in \mathbb{R}$ and $\rho >0$, we consider $\mu(s,\rho)$ the total amount of time that the slice $u_s$ spends under the level $r^+ - \rho$, i.e.,
    $$\mu(s,\rho) = \text{Leb} \{t \mid r(u(s,t)) \leq r^+ - \rho \}.$$
    On this step we want to prove that
    \begin{equation} \label{eq: eq of step 3}
        \mu(s,\rho) \rho \leq \frac{r^+}{A_h(r^+)}E(u).
    \end{equation}
    From the maximum principal \eqref{eq: maximum principle}, we have that
    \begin{equation} \label{eq: step 3 eq 1}
        \int_0^{\tau} r(u_s)dt \leq (\tau - \mu(s,\rho))r^+ + \mu(s,\rho)(r^+ - \rho) = \tau r^+ - \mu(s,\rho)\rho.
    \end{equation}
    Now, by combining \eqref{eq: step 2}, Remark \ref{rmk:rmk inside constraint lemma} and \eqref{eq: step 3 eq 1}, we obtain 
    $$\tau r^+ - \frac{r^+}{A_h(r^+)} E(u) \leq \tau r^+ - \mu(s,\rho)\rho,$$
    and \eqref{eq: eq of step 3} follows. \\

    \textbf{Step 4:} In this step we finish the proof of Theorem \ref{thm: lower bound for the r coordinate of a curve}. We consider $\epsilon>0$ and $C'>0$ be as in the Remark \ref{rmk: partial s u bounded by energy}, and assume that $E(u)< \epsilon$. If $r^- \leq \inf_{\mathbb{R} \times [0,\tau]} r(u)$, then there is nothing to show. If $\inf_{\mathbb{R}\times [0,\tau]} r(u) < r^-$, then fix $0< \eta < r^- - \inf_{\mathbb{R}\times [0,\tau]} r(u)$. Choose $(s_0,t_0) \in \mathbb{R} \times [0,\tau]$ such that $r(u(s_0,t_0))< r^- - \eta$. Notice that by the Bourgeios-Oancea monotonicity property \eqref{eq: Bourgeois-Oancea}, the slice $u_{s_0}$ has to rise at least to the level $r^-$. Now we consider the following claim: \\ \\
    \textit{Claim 2:} For any $t_1 \in [0,\tau]$, with $r(u(s_0,t_1)) \geq  r^- - \eta/2$, we have that 
    $$\eta/2 \leq C'\sqrt{r^+} E(u)^{1/4}|t_1-t_0|.$$
    
    Let us finish the proof of Theorem \ref{thm: lower bound for the r coordinate of a curve} assuming the claim. Consider $\rho = r^+ - r^- + \eta/2$. From the claim and the Bourgeois-Oancea monotonicity \eqref{eq: Bourgeois-Oancea}, we see that
    \begin{equation} \label{eq: eq step 4}
        \eta/2 \leq C' \sqrt{r^+}E(u)^{1/4} \mu(s_0,\rho).
    \end{equation}
    By combining \eqref{eq: eq step 4} and \eqref{eq: eq of step 3}, we obtain
    $$\eta^2/4 \leq (r^+ - r^-)\eta/2 + \eta^2/4 \leq \frac{C'(r^+)^{3/2}}{A_h(r^+)}E(u)^{5/4}.$$    
    Now, by taking the square root we get
    \begin{equation} \label{eq: step 3 second eq}
        \eta \leq \frac{2\sqrt{C'}(r^+)^{3/4}}{\sqrt{A_h(r^+)}}E(u)^{5/8}.
    \end{equation}
    Since \eqref{eq: step 3 second eq} holds for all $\eta<r^- - \inf_{\mathbb{R}\times [0,\tau]}r(u)$, then it holds for $\eta = r^- - \inf_{\mathbb{R}\times [0,\tau]}r(u)$. By setting $C = 2\sqrt{C'}$, we conclude the proof of the Theorem.

    \begin{proof}[Proof of Claim 2]
        We have that 
        \begin{equation} \label{eq: claim 2}
            \eta/2 < r(u(s_0,t_1)) - r(u(s_0,t_0)) = \int_{t_0}^{t_1} dr(\partial_t u(s_0,t))dt.
        \end{equation}
    By using the identity
    $$ dr(\partial_t u) = r(u) \alpha(\partial_s u) + r(u)\alpha (J X_H(u)) = r(u)\alpha(\partial_s u),$$
    which follows from \eqref{eq: relation between dr and alpha}, the Floer equation \eqref{eq: Floer equation with boundary constraint}, and the condition that $\alpha(JX_H) = 0$, the inequality \eqref{eq: claim 2} becomes
    \begin{align*}
         \eta/2 & \leq \int_{t_0}^{t_1}r(\partial_s u(s_0,t))\alpha(\partial_s u(s_0,t))dt \\
         & \leq \int_{t_0}^{t_1} r(u(s_0,t)) \frac{||\partial_s u(s_0,t)||}{||R_{\alpha}(u(s_0,t))||}dt \\
         & \leq \int_{t_0}^{t_1} \sqrt{r(u(s_0,t))}||\partial_s u(s_0, t)||dt \\
         & \leq \sqrt{r^+} \max_{t_0\leq t \leq t_1}||\partial_s u(s_0,t)|| |t_1 - t_0| \\
         & \leq C\sqrt{r^+} E(u)^{1/4}|t_1 - t_0|,
    \end{align*}
        where $R_{\alpha}$ is the Reeb vector field of $\alpha$. In the third inequality, we used the identify 
        $$||R_{\alpha}||^2 = \omega(R_{\alpha},JR_{\alpha}) = r,$$
        in the fourth we used the maximum principle \eqref{eq: maximum principle}, and the last is a consequence of Remark \ref{rmk: partial s u bounded by energy}. 
    \end{proof}
\end{proof}

\begin{proof}[Proof of Bourgeois-Oancea monotonicity \eqref{eq: Bourgeois-Oancea}] By contradiction, we assume that there exists $s_0 \in \mathbb{R}$ such that 
\begin{equation} \label{eq: bougeouis-oancea monotonicity contradiction hyp}
    \max_{t\in[0,\tau]}r(u(s_0,t) < r^-.
\end{equation}
Then it follows from the maximal principal, \eqref{eq: maximum principle}, that 
$$\max_{t\in [0,\tau]} r(u(s,t)) \leq r^-,$$
for all $s\in [s_0,\infty)$. Now from \eqref{eq:variation of r(u_s) is bounded by above} and the fact that $A'_h(r)\geq 0$, we obtain
$$\frac{d}{ds} \int_0^{\tau} r(u_s) dt \leq -\tau A_h(r^-) + \int_0^{\tau} A_h(r(u_s))dt \leq 0,$$
for all $s\in [s_0,\infty)$. Now, from assumption \eqref{eq: bougeouis-oancea monotonicity contradiction hyp}, we have 
$$\int_0^{\tau}r(u(s_0,t)dt) \leq \tau r^- - \epsilon,$$
for some $\epsilon>0$. Then the same holds for $s \in [s_0,\infty)$, and this is a contradiction with the fact that the left hand side converges to $\tau r^-$, which completes the proof of \eqref{eq: Bourgeois-Oancea} 
\end{proof}

\bibliographystyle{alpha}\bibliography{mybibliography}
\end{document}